\documentclass{amsart}
\usepackage{amssymb}
\usepackage[all]{xy}

\theoremstyle{remark}
\newtheorem{example}{Example}[section]
\newtheorem{remark}[example]{Remark}
\theoremstyle{definition}
\newtheorem{definition}[example]{Definition}
\theoremstyle{plain}
\newtheorem{proposition}[example]{Proposition}
\newtheorem{corollary}[example]{Corollary}
\newtheorem{conjecture}[example]{Conjecture}
\newtheorem{theorem}[example]{Theorem}
\newtheorem{lemma}[example]{Lemma}

\DeclareMathOperator{\Tors}{Tors}
\DeclareMathOperator{\Sstab}{Sstab}
\DeclareMathOperator{\Coh}{\mathcal{C}\mathit{\!o\!h}}
\DeclareMathOperator{\Bun}{\mathcal{B}\mathit{\!u\!n}}
\DeclareMathOperator{\Nil}{\mathcal{N}\mathit{\!\!il}}

\newcommand{\Fields}{\mathsf{Fields}}
\newcommand{\Sets}{\mathsf{Sets}}

\newcommand{\E}{\mathcal{E}}
\newcommand{\F}{\mathcal{F}}
\newcommand{\bbP}{\mathbb{P}}
\newcommand{\bbN}{\mathbb{N}}
\newcommand{\bbZ}{\mathbb{Z}}
\newcommand{\bbQ}{\mathbb{Q}}
\newcommand{\calM}{\mathcal{M}}
\newcommand{\calN}{\mathcal{N}}
\newcommand{\calC}{\mathcal{C}}
\newcommand{\calG}{\mathcal{G}}
\newcommand{\calO}{\mathcal{O}}
\newcommand{\calI}{\mathcal{I}}

\newcommand{\frakj}{\mathfrak{j}}
\newcommand{\frakm}{\mathfrak{m}}
\newcommand{\frakn}{\mathfrak{n}}
\newcommand{\fraka}{\mathfrak{a}}

\newcommand{\map}{\bullet \to \bullet}
\newcommand{\sub}{\bullet \hookrightarrow \bullet}
\newcommand{\filt}{\sub \hookrightarrow \bullet}
\newcommand{\subs}{\sub \hookleftarrow \bullet}

\DeclareMathOperator{\im}{im}
\DeclareMathOperator{\coker}{coker}
\DeclareMathOperator{\Spec}{Spec}
\DeclareMathOperator{\Pic}{Pic}
\DeclareMathOperator{\SB}{SB}
\DeclareMathOperator{\Mod}{Mod}
\DeclareMathOperator{\rank}{rank}
\DeclareMathOperator{\ed}{ed}

\DeclareMathOperator{\ind}{ind}
\DeclareMathOperator{\trdeg}{trdeg}
\DeclareMathOperator{\Aut}{Aut}
\DeclareMathOperator{\End}{End}
\DeclareMathOperator{\Hom}{Hom}
\DeclareMathOperator{\Ext}{Ext}
\DeclareMathOperator{\Mat}{Mat}
\DeclareMathOperator{\Res}{Res}
\DeclareMathOperator{\rmH}{H}

\newcommand{\Gm}{\mathbb{G}_{\mathrm{m}}}
\newcommand{\pr}{\mathrm{pr}}

\newcommand{\prsub}{\pr_{\mathrm{sub}}}
\newcommand{\prext}{\pr_{\mathrm{ext}}}
\newcommand{\prcap}{\pr_{\cap}}
\newcommand{\prto}{\pr_{\to}}
\newcommand{\id}{\mathrm{id}}

\begin{document}

\title[Essential dimension of coherent sheaves]{On the essential dimension of coherent sheaves}

\author[I. Biswas]{Indranil Biswas}
\address{School of Mathematics, Tata Institute of Fundamental Research, Homi Bhabha Road, Bombay 400005, India}
\email{indranil@math.tifr.res.in}

\author[A. Dhillon]{Ajneet Dhillon}
\address{Department of Mathematics, University of Western Ontario, London, Ontario N6A 5B7, Canada}
\email{adhill3@uwo.ca}

\author[N. Hoffmann]{Norbert Hoffmann}
\address{Department of Mathematics and Computer Studies, Mary Immaculate College, South Circular Road,
  Limerick, Ireland}
\email{norbert.hoffmann@mic.ul.ie}
\thanks{I.\ B. is supported by the J.\ C.\ Bose Fellowship. A.\ D. is partially supported by NSERC.
  N.\ H. was partially supported by SFB 647: Space - Time - Matter in Berlin.
  He thanks TIFR Bombay for hospitality, and Bernd Kreussler
  for a useful discussion on bundles over elliptic curves.}



\begin{abstract}
  We characterize all fields of definition for a given coherent sheaf over a projective scheme in terms of
  projective modules over a finite-dimen\-sional endomorphism algebra. This yields general results on the
  essential dimension of such sheaves. Applying them to vector bundles over a smooth projective curve $C$, we
  obtain an upper bound for the essential dimension of their moduli stack. The upper bound is sharp if the
  conjecture of Colliot-Th\'el\`ene, Karpenko and Merkurjev holds. We find that the genericity property proved
  for Deligne-Mumford stacks by Brosnan, Reichstein and Vistoli still holds for this Artin stack,
  unless the curve $C$ is elliptic.
\end{abstract}

\maketitle

\section{Introduction}
The essential dimension of an algebraic object was introduced in \cite{br:97}. Roughly speaking,
it is the number of algebraically independent parameters needed to define the object; the
precise definition is recalled below. This notion has been studied intensively, leading to
many interesting connections with several areas of algebra and algebraic geometry, as the recent
surveys \cite{reichstein:11} and \cite{merkurjev:13} show.

The essential dimension of a moduli stack is the supremum of the essential dimensions of the
objects it parameterizes. For smooth Deligne-Mumford stacks, it suffices to consider generic objects,
according to the genericity theorem of Brosnan, Reichstein and Vistoli \cite[Theorem 6.1]{brv:11}. They use it
to determine the essential dimension of the moduli stack of curves. In an appendix to
\cite{brv:11}, Fakhruddin does likewise for the moduli stack of abelian varieties. The genericity theorem
is generalized to smooth Artin stacks with reductive automorphism groups in \cite{rv:11}.

The subject of this article is the essential dimension of coherent sheaves over a projective scheme.
We relate it to the essential dimension of projective modules over a finite-dimensional algebra, and study
the latter systematically. The essential dimension also involves the number of moduli. In order to count
moduli of coherent sheaves, we express those with nilpotent endomorphisms as iterated extensions.

We then apply our general results to the special case of vector bundles of fixed rank $r$ and degree $d$ over
a smooth projective curve $C$. Theorem \ref{thm:ed_vb} gives the essential dimension of the moduli
stack $\Bun_{C, r, d}$ in this case, modulo the now famous conjecture of Colliot-Th\'el\`ene, Karpenko and
Merkurjev \cite[\S 1]{ckm:07}. Our result improves the upper bounds on this essential dimension given in
\cite{dl:09} and in \cite{bdl:12}.

The stack $\Bun_{C, r, d}$ is not Deligne-Mumford, and its automorphism groups are in general not reductive.
We find that $\Bun_{C, r, d}$ nevertheless has the genericity property mentioned above, unless the curve $C$
is elliptic. In the case of an elliptic curve $C$, Proposition \ref{prop:g=1} gives counterexamples.
Our methods specifically address non-reductive automorphism groups, by focussing on nilpotent endomorphisms.

Let $X \hookrightarrow \bbP^N_k$ be a projective scheme over a base field $k$. Let $K$ be a field
containing $k$, and let $E$ be a coherent sheaf over the base change $X_K := X \otimes_k K$.
One says that $E$ is \emph{defined over a field $K'$} with $k \subseteq K' \subseteq K$ if there is a
coherent sheaf $E'$ over $X_{K'}$ with $E' \otimes_{K'} K \cong E$. The \emph{essential
dimension} of $E$ is
\begin{equation*}
  \ed_k( E) := \min_{K'} \trdeg_k K'
\end{equation*}
where the minimum is taken over all fields $K'$ with $k \subseteq K' \subseteq K$ such that $E$ is defined
over $K'$. 

Let $k( E) \subseteq K$ denote the \emph{field of moduli} for the coherent sheaf $E$ over $X_K$; cf.
Remark \ref{rem:moduli}. Since $k( E) \subseteq K'$ for every field of definition $K'$ for $E$, we have
\begin{equation*}
  \ed_k( E) = \trdeg_k k( E) + \ed_{k( E)}( E)
\end{equation*}
where $\ed_{k( E)}( E)$ refers to the scheme $X_{k( E)}$ over $k( E)$ instead of $X$ over $k$.

The essential dimension of $E$ over $k( E)$ measures how far $E$ is from being defined over $k( E)$.
This defect is caused by $\Aut( E)$, since an object without automorphisms is usually defined over its field
of moduli. We make use of the fact that $\Aut( E)$ is the group of units of the finite-dimensional $K$-algebra
$\End( E)$. Theorem \ref{thm:equivalence} describes the obstruction against defining $E$ over $k( E)$ in terms
of modules over such algebras. This is the basis of our results on the essential dimension of $E$ over $k(E)$.
We also deduce that every vector bundle over an elliptic curve is defined over its field of moduli.

We then have to estimate the transcendence degree of $k( E)$. This is more subtle if $E$ has nilpotent
endomorphism. Our estimates are based on Theorem \ref{thm:Nil_X}, which describes sheaves with nilpotent
endomorphisms as iterated extensions.

These two theorems are our main technical tools.
We formulate them for coherent sheaves over projective schemes, but the method generalizes
to other kinds of objects as long as they have finite-dimensional endomorphism algebras.

The structure of this paper is as follows. Section \ref{sec:modules} deals with projective modules over
right-artinian rings, in particular over finite-dimensional algebras. Section \ref{sec:ed+algs} studies
the essential dimension of such modules, and reduces this question to the case of central simple algebras,
which is well-studied.

Section \ref{sec:End} deals with endomorphism algebras of coherent sheaves.
Section \ref{sec:fields_of_def} relates the fields of definition for $E$ to those for some
module over an endomorphism algebra, and deduces information on the essential dimension of $E$ over $k( E)$.

Section \ref{sec:moduli} contains the moduli count, in particular for sheaves with nilpotent endomorphisms.
Section \ref{sec:bundles} puts all this together in the case of vector bundles over a curve,
and contains our results on their essential dimension.

\section{Projective Modules over Right-Artinian Rings} \label{sec:modules}
Let $R$ be a ring. Our rings are always associative, and they always have a unit, but they are
not necessarily commutative. By an $R$-module, we mean a right $R$-module, unless stated otherwise.
Let $\frakn \subset R$ be a nilpotent two-sided ideal.
\begin{lemma} \label{lem:lift}
  Every element $q \in R/\frakn$ with $q^2 = q$ admits a lift $p \in R$ with $p^2 = p$.
\end{lemma}
\begin{proof}
  By assumption, there is an integer $n \geq 1$ such that $\frakn^n = 0$. Using induction over $n$,
  we may assume $\frakn^2 = 0$ without loss of generality.

  Let $p \in R$ be any lift of $q$. Then $p^2 \equiv p$ modulo $\frakn$, and hence $(p^2 - p)^2 = 0$.
  Therefore, $p' := -2p^3 + 3p^2 \in R$ is another lift of $q$, and
  \begin{equation*}
    (p')^2 = 4p^6 - 12p^5 + 9p^4 = (p^2 - p)^2( 4p^2 - 4p - 3) - 2p^3 + 3p^2 = p'.\qedhere
  \end{equation*}
\end{proof}
\begin{corollary} \label{cor:lift}
  Let $N$ be a finitely generated projective $(R/\frakn)$-module.
  Then there is a finitely generated projective $R$-module $M$ such that $M/M \frakn \cong N$.
  The finitely generated projective $R$-module $M$ is unique up to isomorphisms.
\end{corollary}
\begin{proof}
  By assumption, $N$ is isomorphic to a direct summand of a free module $(R/\frakn)^r$ for some $r \in \bbN$.
  Therefore, $N$ is isomorphic to the image of a matrix
  \begin{equation*}
    q \in \Mat_{r \times r}( R/\frakn)
  \end{equation*}
  with $q^2 = q$. Using Lemma \ref{lem:lift}, we can lift $q$ to a matrix
  \begin{equation*}
    p \in \Mat_{r \times r}( R)
  \end{equation*}
  with $p^2 = p$. The image of $p$ is a finitely generated projective $R$-module $M$ with $M/M \frakn \cong N$.
  For the uniqueness, suppose that $M'$ is another finitely generated projective $R$-module with
  $M'/M'\frakn \cong N$. Then there are $(R/\frakn)$-linear maps
  \begin{equation*}
    g_1: M/M \frakn \longrightarrow M'/M'\frakn \qquad\text{and}\qquad
    g_2: M'/M'\frakn \longrightarrow M/M \frakn
  \end{equation*}
  with $g_1 \circ g_2 = \id$ and $g_2 \circ g_1 = \id$.
  Since $M$ and $M'$ are direct summands of free modules, we can lift $g_1$ and $g_2$ to $R$-linear maps
  \begin{equation*}
    f_1: M \longrightarrow M' \qquad\qquad\text{and}\qquad\qquad f_2: M' \longrightarrow M.
  \end{equation*}
  They satisfy $f_1 \circ f_2 \equiv \id$ and $f_2 \circ f_1 \equiv \id$ modulo $\frakn$.
  Therefore, $f_1 \circ f_2$ and $f_2 \circ f_1$ are automorphisms. This shows that $M'$ is isomorphic to $M$.
\end{proof}
We will only need rings that are finite-dimensional algebras over a field.
These satisfy the descending chain condition for right ideals, so they are right-artinian.
\begin{definition}
  A projective module $M$ over a right-artinian ring $R$ has \emph{rank} $r \in \bbQ_{> 0}$ if
  the direct sum $M^{\oplus n}$ is free of rank $nr$ for some $n \in \bbN$ with $nr \in \bbN$.
\end{definition}
Let $R$ be a right-artinian ring. Let $\frakj \subset R$ denote the Jacobson radical; this is
the smallest two-sided ideal such that $R/\frakj$ is semisimple. The ideal $\frakj \subseteq R$
is known to be nilpotent (see for example Theorem 14.2 in \cite{isaacs:93}). Wedderburn's theorem states
\begin{equation*}
  R/\frakj \cong \Mat_{n_1 \times n_1}( D_1) \times \cdots \times \Mat_{n_s \times n_s}( D_s)
\end{equation*}
for some division rings $D_1, \ldots, D_s$ and some integers $n_1, \ldots, n_s \geq 1$. We put
\begin{equation*}
  d_R := \gcd( n_1, \ldots, n_s).
\end{equation*}
Corollary \ref{cor:lift} states that there is a finitely generated projective $R$-module $M_R$ with
\begin{equation*}
  M_R/M_R \frakj \cong \Mat_{(n_1/d_R) \times n_1}( D_1) \times \cdots \times \Mat_{(n_s/d_R) \times n_s}( D_s),
\end{equation*}
and that $M_R$ is unique up to isomorphisms. Clearly $M_R$ is projective of rank $1/d_R$. 
\begin{proposition} \label{prop:M_R}
  Let $R$, $d_R$ and $M_R$ be as above. If $M$ is a projective $R$-module of some rank $r \in \bbQ_{> 0}$,
  then $r = n/d_R$ and $M \cong M_R^{\oplus n}$ for some integer $n \geq 1$.
\end{proposition}
\begin{proof}
  Suppose that $M$ is a projective module of rank $r$ over $R$. Then $M/M \frakj$ is a projective module
  of rank $r$ over $R/\frakj$. Hence we conclude
  \begin{equation*}
    M/M \frakj \cong \Mat_{n_1 r \times n_1}( D_1) \times \cdots \times \Mat_{n_s r \times n_s}( D_s)
  \end{equation*}
  with $n_1 r, \ldots, n_s r \in \bbN$. In other words,
  $r = n/d_R$ for some integer $n \geq 1$, and $M/M \frakj \cong (M_R/M_R \frakj)^{\oplus n}$.
  This implies $M \cong M_R^{\oplus n}$ due to Corollary \ref{cor:lift}.
\end{proof}

\section{Essential Dimension and Finite-Dimensional Algebras} \label{sec:ed+algs}
Let $k$ be a field. Let $\Fields/k$ denote the category of fields $K \supseteq k$. Let a functor
\begin{equation*}
  F: \Fields/k \longrightarrow \Sets
\end{equation*}
be given. If an element $a \in F( K)$ is the image of an element $a' \in F( K')$ for some intermediate field
$k \subseteq K' \subseteq K$, then $a$ is said to be \emph{defined over $K'$}.
\begin{definition}[Merkurjev] ${}_{}$
  \begin{itemize}
   \item[i)] The \emph{essential dimension of an element $a \in F( K)$} is
    \begin{equation*}
      \ed_k( a) := \inf_{K'} \trdeg_k K'
    \end{equation*}
    where the infimum is over all fields $k \subseteq K' \subseteq K$ over which $a$ is defined. 
   \item[ii)] The \emph{essential dimension of the functor $F$} is
    \begin{equation*}
      \ed_k( F) := \sup_a \ed_k( a)
    \end{equation*}
    where the supremum is over all fields $K \supseteq k$ and all elements $a \in F( K)$. 
    We put $\ed_k( F) = -\infty$ if $F( K) = \emptyset$ for all $K$.
   \item[iii)] The \emph{essential dimension of a stack $\calM$ over $k$}
    is the essential dimension of the functor $\Fields/k \longrightarrow \Sets$
    that sends each field $K \supseteq k$ to the set of isomorphism classes in the groupoid $\calM( K)$.
  \end{itemize}
\end{definition}
Given a finite-dimensional $k$-algebra $A$ and a number $r \in \bbQ_{> 0}$, we denote by
\begin{equation*}
  \Mod_{A, r}: \Fields_k \longrightarrow \Sets
\end{equation*}
the functor that sends each field $K \supseteq k$ to the set $\Mod_{A, r}( K)$ of isomorphism classes of
projective $(A \otimes_k K)$-modules of rank $r$. Each of these sets $\Mod_{A, r}( K)$ has at most one
element by Proposition \ref{prop:M_R}.

This section deals with
\begin{equation*}
  \ed_k( \Mod_{A, r}).
\end{equation*}
The following three propositions will allow us to assume that $A$ is semisimple, simple,
and a division algebra, respectively.
\begin{proposition} \label{prop:reduce}
  If $\frakn \subset A$ is a nilpotent two-sided ideal, then
  \begin{equation*}
    \ed_k( \Mod_{A, r}) = \ed_k( \Mod_{A/\frakn, r}).
  \end{equation*}
\end{proposition}
\begin{proof}
  Corollary \ref{cor:lift} states that the canonical map
  \begin{equation*}
    \Mod_{A, r}( K) \longrightarrow \Mod_{A/\frakn, r}( K)
  \end{equation*}
  is bijective for every field $K \supseteq k$.
\end{proof}
\begin{proposition} \label{prop:product}
  If $A$ is isomorphic to a product of $k$-algebras $A_i$, then
  \begin{equation*}
    \ed_k( \Mod_{A, r}) \leq \sum_i \ed_k( \Mod_{A_i, r}).
  \end{equation*}
\end{proposition}
\begin{proof}
  For each field $K \supseteq k$, we have a canonical bijection
  \begin{equation*}
    \prod_i \Mod_{A_i, r}( K) \longrightarrow \Mod_{A, r}( K)
  \end{equation*}
  which sends each sequence of projective $(A_i \otimes_k K)$-modules $M_i$ to their product $M$.
  If each $M_i$ is defined over some intermediate field $k \subseteq K_i' \subseteq K$, then $M$ is defined
  over the compositum $K' \subseteq K$ of all $K_i'$. This shows $\ed_k( M) \leq \sum_i \ed_k( M_i)$.
\end{proof}
\begin{proposition} \label{prop:morita}
  If $A \cong \Mat_{n \times n}( B)$ for a $k$-algebra $B$, then 
  \begin{equation*}
    \ed_k( \Mod_{A, r}) = \ed_k( \Mod_{B, nr}).
  \end{equation*}
\end{proposition}
\begin{proof}
  For each field $K \supseteq k$, we have a canonical bijection
  \begin{equation*}
    \Mod_{A, r}( K) \longrightarrow \Mod_{B, nr}( K)
  \end{equation*}
  which sends a projective $(A \otimes_k K)$-module $M$ to $M \otimes_A \Mat_{n \times 1}( B)$.
  The inverse map sends a projective $(B \otimes_k K)$-module $N$ to $N \otimes_B \Mat_{1 \times n}( B)$.
\end{proof}
\begin{proposition} \label{prop:1/d}
  If $r = n/d$ for coprime integers $n, d \geq 1$, then
  \begin{equation*}
    \ed_k( \Mod_{A, r}) = \ed_k( \Mod_{A, 1/d}).
  \end{equation*}
\end{proposition}
\begin{proof}
  For each field $K \supseteq k$, we have a canonical map
  \begin{equation*}
    \Mod_{A, 1/d}( K) \longrightarrow \Mod_{A, r}( K)
  \end{equation*}
  which sends a module $M$ to $M^{\oplus n}$. This map is bijective due to Proposition \ref{prop:M_R}.
\end{proof}
Let $A$ be a simple $k$-algebra, with center $l \supseteq k$. Recall that the \emph{degree} of $A$ is
\begin{equation*}
  \deg A := \sqrt{\dim_l A} \in \bbN.
\end{equation*}
If $A \cong \Mat_{n \times n}( D)$ for a central division algebra $D$ over $l$, then the \emph{index} of $A$ is
\begin{equation*}
  \ind A := \deg D = \deg A/n.
\end{equation*}

In the case of a central simple $k$-algebra $A$, an upper bound for $\ed_k( \Mod_{A, 1/\deg A})$ is proved
in \cite[\S1]{ckm:07}. We follow the argument and generalize it.
\begin{proposition} \label{prop:bound}
  If $A$ is a simple $k$-algebra, and $0 < r < 1$, then
  \begin{equation*}
    \ed_k( \Mod_{A, r}) \leq r(1-r) \dim_k A.
  \end{equation*}
  If moreover $r \deg A \not\in \bbN$, then $\ed_k( \Mod_{A, r}) = - \infty$.
\end{proposition}
\begin{proof}
  Let $l \supseteq k$ denote the center of $A$. Then $A$ is a central simple $l$-algebra. Let
  \begin{equation*}
    \SB( r, A)
  \end{equation*}
  denote the generalized Severi-Brauer variety over $l$ that parameterizes right ideals $\fraka \subset A$
  which are projective of rank $r$ over $A$. This variety is a form of the Grassmannian that parameterizes
  linear subspaces of dimension $r \deg A$ in a vector space of dimension $\deg A$. Therefore,
  $\SB( r, A) = \emptyset$ if $r \deg A \not\in \bbN$, and otherwise
  \begin{equation*}
    \dim_l \SB( r, A) = r(1-r) \dim_l A.
  \end{equation*}

  We denote the Weil restriction of the variety $\SB( r, A)$ from $l$ to $k$ by
  \begin{equation*}
    \Res_{l/k} \SB( r, A).
  \end{equation*}
  This $k$-scheme represents by definition the functor that sends each $k$-scheme $S$ to the set of
  $l$-morphisms from $S \otimes_k l$ to $\SB( r, A)$. This functor is indeed representable,
  for example by Theorem 7.6/4 in \cite{blr:90}. We have
  \begin{equation*}
    \dim_k \Res_{l/k} \SB( r, A) = [l:k] \cdot \dim_l \SB( r, A).
  \end{equation*}
 
  Now let $K \supseteq k$ be a field such that
  \begin{equation*}
    \Mod_{A, r}( K) \neq \emptyset.
  \end{equation*}
  Due to Proposition \ref{prop:M_R}, there then is a right ideal $\fraka \subset A \otimes_k K$ which is
  projective of rank $r$ over $A$. The ideal $\fraka$ corresponds to a $K$-valued point in $\Res_{l/k}
  \SB( r, A)$. Let $K' \subseteq K$ be the residue field of that point in $\Res_{l/k} \SB( r, A)$. Then
  \begin{equation*}
    \Mod_{A, r}( K') \neq \emptyset
  \end{equation*}
  because the ideal $\fraka$ is already defined over $K'$, and
  \begin{equation*}
    \trdeg_k( K') \leq \dim_k \Res_{l/k} \SB( r, A) = r( 1-r) \dim_k A. \qedhere
  \end{equation*}
\end{proof} 

\begin{corollary} \label{cor:ed_bound}
  If $A \cong \Mat_{n \times n}( B)$ for a simple $k$-algebra $B$, then
  \begin{equation*}
    \ed_k( \Mod_{A, r}) < n r \dim_k B.
  \end{equation*}
\end{corollary}
\begin{proof}
  Proposition \ref{prop:morita}, Proposition \ref{prop:1/d} and Proposition \ref{prop:bound} yield
  \begin{equation*}
    \ed_k( \Mod_{A, r}) = \ed_k( \Mod_{B, nr}) = \ed_k( \Mod_{B, 1/d})
      \leq \frac{1}{d}( 1 - \frac{1}{d}) \dim_k B\, ,
  \end{equation*}
  where $d \in \bbN$ is the denominator of $n r \in \bbQ_{> 0}$. Since $1/d \leq nr$, we get
  \begin{equation*}
    \ed_k( \Mod_{A, r}) < n r \dim_k B. \qedhere
  \end{equation*}
\end{proof}

Given a prime number $p$ and an integer $n \geq 1$, we denote by $v_p( n)$ the $p$-adic valuation of $n$.
Therefore, $p^{v_p( n)}$ is the largest power of $p$ that divides $n$.

\begin{corollary} \label{cor:ed_D}
  If $D$ is a division algebra over $k$, and $d$ divides $\deg D$, then
  \begin{equation*}
    \ed_k( \Mod_{D, 1/d}) \leq [l:k] \sum_{p|\deg D} p^{2v_p( \deg D/d)}( p^{v_p( d)} - 1)
  \end{equation*}
  where $l \supseteq k$ denotes the center of $D$.
\end{corollary}
\begin{proof}
  There are central division algebras $D_p$ over $l$ such that
  \begin{equation*}
    D \cong \bigotimes_{p|\deg D} D_p
  \end{equation*}
  over $l$, and $\deg D_p = p^{v_p( \deg D)}$. We put $d_p := p^{v_p( d)}$.

  Let a field $K \supseteq k$ be given. Let $\frakm_i$ denote the maximal ideals in $l \otimes_k K$, and put
  \begin{equation*}
    L_i = (l \otimes_k K)/\frakm_i.
  \end{equation*}
  Since $D \otimes_l L_i$ and $D_p \otimes_l L_i$ are central simple $L_i$-algebras, they are simple as
  $K$-algebras. They are precisely the simple quotients of $D \otimes_k K$ and of $D_p \otimes_k K$,
  respectively. Using Proposition \ref{prop:M_R}, we conclude that the canonical map
  \begin{equation*}
    \prod_{p|\deg D} \Mod_{D_p, 1/d_p}( K) \longrightarrow \Mod_{D, 1/d}( K)
  \end{equation*}
  is bijective for every field $K \supseteq k$. This proves the inequality
  \begin{equation*}
    \ed_k( \Mod_{D, 1/d}) \leq \sum_{p|\deg D} \ed_k( \Mod_{D_p, 1/d_p}).
  \end{equation*}
  Using Proposition \ref{prop:bound} to bound each summand from above, the result follows.
\end{proof}

Karpenko \cite[Theorem 4.3]{karpenko:09} has proved that this bound is sharp if $D$ has center $k$,
and $\deg D$ is a prime power:
\begin{theorem}[Karpenko] ${}_{}$\\\indent
  If $D$ is a central division algebra over $k$ with $\deg D = p^n$, and $1 \leq m \leq n$, then
  \begin{equation*}
    \ed_k( \Mod_{D, 1/p^m}) = p^{2(n-m)}( p^m - 1).
  \end{equation*}
\end{theorem}

Colliot-Th\'el\`ene, Karpenko and Merkurjev \cite[\S1]{ckm:07} have conjectured that
the above bound is sharp if $D$ has center $k$, and $d = \deg D$:
\begin{conjecture}[Colliot-Th\'el\`ene, Karpenko, Merkurjev] \label{conj:CKM} ${}_{}$\\\indent
  If $D$ is a central division algebra over $k$, then
  \begin{equation*}
    \ed_k( \Mod_{D, 1/\deg D}) = \sum_{p|\deg D} ( p^{v_p( \deg D)} - 1).
  \end{equation*}
\end{conjecture}

\section{Endomorphism Algebras of Coherent Sheaves} \label{sec:End}
Let $X \hookrightarrow \bbP^N_k$ be a projective scheme over the base field $k$.
We put $X_S := X \times_k S$ for each $k$-scheme $S$, and $X_K := X \otimes_k K$ for each field $K \supseteq k$.
Let $E$ be a coherent sheaf over $X_K$. Its endomorphism algebra $\End( E)$ satisfies
\begin{equation*}
  \dim_K \End( E) < \infty,
\end{equation*}
since $\End( E)$ is the space of the global sections of the coherent sheaf of endomorphisms of $E$.
Therefore, the theory of finite-dimensional algebras applies to $\End( E)$.
Let $\frakj( E)$ denote the Jacobson radical of $\End( E)$. Wedderburn's Theorem states
\begin{equation} \label{eq:wedderburn}
  \End( E)/\frakj( E) \cong \prod_i \Mat_{n_i \times n_i}( D_i)
\end{equation}
for some finite-dimensional division algebras $D_i$ over $K$ and some integers $n_i \geq 1$.

A nonzero coherent sheaf $E$ over $X_K$ is called \emph{indecomposable} if $E \cong E' \oplus E''$
implies that either $E' = 0$ or $E'' = 0$. Then $\End( E)/\frakj( E)$ is a division ring $D$,
according to Lemma 6 in \cite{atiyah:56}. We will use the following slightly more general fact. 
\begin{lemma} \label{lem:decompose}
  In the notation of \eqref{eq:wedderburn}, the coherent sheaf $E$ admits a decomposition
  \begin{equation*}
    E \cong \bigoplus_i E_i^{\oplus n_i}
  \end{equation*}
  into indecomposable coherent sheaves $E_i$ with $\End( E_i)/\frakj( E_i) \cong D_i$.
\end{lemma}
\begin{proof}
  Suppose that there is an isomorphism of $K$-algebras
  \begin{equation*}
    \End( E)/\frakj( E) \cong A' \times A''.
  \end{equation*}
  Then $(1, 0) \in A' \times A''$ corresponds to an element $q \in \End( E)/\frakj( E)$ with $q^2 = q$.
  Lemma \ref{lem:lift} allows us to lift $q$ to an element $p \in \End( E)$ with $p^2 = p$. Therefore,
  \begin{equation*}
    E = E' \oplus E''
  \end{equation*}
  with $E' := \im p \subseteq E$ and $E'' := \im( 1-p) \subseteq E$. We have
  \begin{equation*}
    \End( E')/\frakj( E') \cong A' \quad\text{and}\quad \End( E'')/\frakj( E'') \cong A'',
  \end{equation*}
  since $\End( E') = p\End( E)p$ and $\End( E'') = (1-p)\End( E)(1-p)$.

  The above argument allows us to assume that $\End( E)/\frakj( E)$ is simple, say
  \begin{equation*}
    \End( E)/\frakj( E) \cong \Mat_{n \times n}( D).
  \end{equation*}
  Corollary \ref{cor:lift} allows us to lift the projective module $\Mat_{1 \times n}( D)$ over this algebra
  to a projective module $M$ of rank $1/n$ over $\End( E)$. The coherent sheaf
  \begin{equation*}
    E_1 := M \otimes_{\End( E)} E
  \end{equation*}
  over $X_K$ satisfies $E_1^{\oplus n} \cong E$, and therefore $\End( E_1)/\frakj( E_1) \cong D$.
  The latter implies that $E_1$ is indecomposable.
\end{proof}

\begin{lemma} \label{lem:connected}
  Suppose that the scheme $X$ is connected and has a rational point $P \in X( k)$.
  Let $E$ be an indecomposable vector bundle over $X_K$. Then we have
  \begin{equation*}
    \dim_K \End( E)/\frakj( E) \leq \rank( E).
  \end{equation*}
\end{lemma}
\begin{proof}
  Note that $X_K$ is still connected, because each connected component of it contains the point $P$.
  Therefore, the rank of $E$ is constant over $X_K$.

  Since $E$ is indecomposable, $\End( E)/\frakj( E)$ is a division algebra $D$ by Lemma 6 in \cite{atiyah:56}.
  The fiber $E_P$ of $E$ at $P$ is a nonzero left module over $\End( E)$, and hence
  \begin{equation*}
    \dim_K( D) \leq \dim_K( E_P) = \rank( E). \qedhere
  \end{equation*}
\end{proof}
Recall that the projective $k$-scheme $X$ is an \emph{elliptic curve} if $X$ is a connected smooth curve
of genus one with a rational point $P \in X( k)$.

Let $\Sstab_{\mu}( X)$ denote the category of semistable vector bundles over $X$ of fixed slope
$\mu \in \bbQ$, and let $\Tors( X)$ denote the category of coherent torsion sheaves over $X$.
We will use the following version of the Fourier-Mukai transform in \cite{polishchuk:03}.
\begin{theorem} \label{thm:FM}
  Suppose that $X$ is an elliptic curve over the base field $k$. There is a natural equivalence of categories
  \begin{equation*}
    T: \Sstab_{\mu}( X) \longrightarrow \Tors( X).
  \end{equation*}
  For any semistable vector bundle $E$ of rank $r$ and degree $d$ over $X$, one has
  \begin{equation*}
    \dim_k \rmH^0( X, T( E)) = \gcd( r, d).
  \end{equation*}
\end{theorem}
\begin{proof}
  See Theorem 14.7, and the remark following it, in \cite{polishchuk:03}.
\end{proof}

\begin{proposition} \label{prop:elliptic}
  Suppose that $X$ is an elliptic curve. Let $E$ be an indecomposable vector bundle over $X_K$.
  Then $\End( E)/\frakj( E)$ is a commutative field.
\end{proposition}
\begin{proof}
  Because the Harder-Narasimhan filtration of any vector bundle over the elliptic curve $X_K$ splits,
  and $E$ is indecomposable, it follows that $E$ is semistable.

  Taking $K$ as base field, the equivalence $T$ in Theorem \ref{thm:FM} maps $E$ to an indecomposable coherent
  torsion sheaf $T( E)$ over $X_K$. Any such sheaf $T( E)$ satisfies
  \begin{equation*}
    T( E) \cong \calO_{X_K}/\calI_x^n
  \end{equation*}
  for some integer $n \geq 1$ and some closed point $x \in X_K$, where $\calI_x \subset \calO_{X_K}$ denotes
  the ideal sheaf of $x$. As the endomorphism algebras of $T( E)$ and of $E$ are isomorphic,
  it follows that $\End( E)/\frakj( E)$ is isomorphic to the residue field of the point $x$.
\end{proof}

\section{Fields of Definition for Coherent Sheaves} \label{sec:fields_of_def}

As before, $X$ is a projective scheme over a field $k$. Let $\Coh_X$ denote the moduli stack of coherent sheaves
over $X$ (cf.\ \cite{lepotier:97} and \cite{hl:10} for moduli spaces of sheaves).
The stack $\Coh_X$ is given by the following groupoid $\Coh_X( S)$ for each $k$-scheme $S$:
\begin{itemize}
 \item An object in $\Coh_X( S)$ is a coherent sheaf $\E$ over $X_S$ which is flat over $S$.
 \item A morphism in $\Coh_X( S)$ is an isomorphism of coherent sheaves.
\end{itemize}
Th\'{e}or\`{e}me 4.6.2.1 in \cite{lm:00} states that $\Coh_X$ is an Artin stack, and that it is locally
of finite type over $k$ (cf.\ also \cite{gomez:01} or \cite{hoffmann:10} for the curve case).

We consider a point of $\Coh_X$, in the sense of \cite[Section 5]{lm:00}. Let $\calG$ be the residue gerbe
of this point, and let $k( \calG)$ denote its residue field. Th\'{e}or\`{e}me 11.3 in \cite{lm:00} states
that $\calG$ is an Artin stack of finite type over the field $k( \calG)$.
\begin{remark} \label{rem:moduli}
  Any coherent sheaf $E$ over $X_K$ for a field $K \supseteq k$ defines a point of $\Coh_X$.
  The residue gerbe $\calG$ of this point parameterizes forms of $E$. The residue field $k( \calG) \subseteq K$
  is known as the field of moduli for $E$. It is also denoted by $k( E)$.
\end{remark}

As before, let $\calG$ be a residue gerbe of $\Coh_X$, with residue field $k( \calG)$.
Hilbert's Nullstellensatz allows us to choose a field extension $l \supseteq k( \calG)$ with
\begin{equation} \label{eq:d}
  d := [l : k( \calG)] < \infty
\end{equation}
such that $\calG( l) \neq \emptyset$. We choose a coherent sheaf $F$ over $X_l$ which is an object in the
groupoid $\calG( l)$. Denoting by $\pi: X_l \twoheadrightarrow X_{k( \calG)}$ the canonical projection, we put
\begin{equation} \label{eq:A}
  A := \End( \pi_* F).
\end{equation}
This section will relate the residue gerbe $\calG$ to the endomorphism algebra $A$.
\begin{example} \label{ex:Gm-gerbe}
  A coherent sheaf $E$ over $X_K$ for some field $K \supseteq k$ is called \emph{simple} if
  \begin{equation*}
    \End( E) = K.
  \end{equation*}
  Let $\calG$ be a residue gerbe of $\Coh_X$ that parameterizes simple sheaves. Then
  $\calG$ is a gerbe with band $\Gm$ over $k( \calG)$, and $A$ is a central simple algebra of degree $d$
  over $k( \calG)$. Both define the same element in the Brauer group of $k( \calG)$.
\end{example}

\begin{theorem} \label{thm:equivalence}
  In the situation preceding the example, consider a field $K \supseteq k( \calG)$.
  Then the following two categories are equivalent:
  \begin{itemize}
   \item the category of coherent sheaves $E$ over $X_K$ which are objects in $\calG( K)$, and
   \item the category of projective modules $M$ of rank $1/d$ over $A_K := A \otimes_{k( \calG)} K$.  
  \end{itemize}
\end{theorem}
\begin{proof}
  We will describe mutually inverse functors between these two categories.

  In one direction, we send a coherent sheaf $E$ over $X_K$ to the $A_K$-module
  \begin{equation*}
    M := \Hom \big( ( \pi_* F) \otimes_{k( \calG)} K, E \big).
  \end{equation*}
  Suppose that $E$ is an object in $\calG( K)$. As $\calG$ is a gerbe over $k( \calG)$, 
  there is a field extension $L \supseteq k( \calG)$ containing $l$ and $K$ such that
  $E \otimes_K L$ and $F \otimes_l L$ are isomorphic over $X_L$;
  we may assume that $[L:K] < \infty$. From this we conclude
  \begin{equation*}
    M \otimes_K L \cong \Hom \big( ( \pi_* F) \otimes_{k( \calG)} L, F \otimes_l L \big).
  \end{equation*}
  Therefore the $A_L$-module $(M \otimes_K L)^{\oplus d}$ is free of rank one, by the projection
  formula. Consequently, its underlying $A_K$-module $M^{\oplus d \cdot [L:K]}$ is free of rank $[L:K]$.
  This shows that the $A_K$-module $M$ is projective of rank $1/d$.

  In the opposite direction, we send an $A_K$-module $M$ to the quasicoherent sheaf
  \begin{equation*}
    E := M \otimes_A \pi_* F
  \end{equation*}
  over $X_K$. Suppose that $M$ is projective of rank $1/d$. Then $E$ is coherent.
  We choose a field extension $L \supseteq k( \calG)$ containing $l$ and $K$.
  Proposition \ref{prop:M_R} implies
  \begin{equation*}
    M \otimes_K L \cong \Hom \big( ( \pi_* F) \otimes_{k( \calG)} L, F \otimes_l L \big),
  \end{equation*}
  since both are projective $A_L$-modules of rank $1/d$. Therefore $E \otimes_K L$ and $F \otimes_l L$
  are isomorphic as coherent sheaves over $X_L$. Hence $E$ is an object in $\calG( K)$. 
\end{proof}
\begin{corollary} \label{cor:ed(G)}
  Let $\calG$ be a residue gerbe of the moduli stack $\Coh_X$ as above.
  \begin{itemize}
   \item[i)] Given a field $K \supseteq k( \calG)$, all objects in the groupoid $\calG( K)$ are isomorphic.
   \item[ii)] For the $k( \calG)$-algebra $A$ in \eqref{eq:A} and the integer $d$ in \eqref{eq:d}, we have
    \begin{equation*}
      \ed_{k( \calG)}( \calG) = \ed_{k( \calG)}( \Mod_{A, 1/d}).
    \end{equation*}
  \end{itemize}
\end{corollary}

\begin{proof}
  Just combine Theorem \ref{thm:equivalence} with Proposition \ref{prop:M_R}. 
\end{proof}

\begin{corollary} \label{cor:vb_bound}
  Suppose that $X$ is connected and has a $k$-rational point.
  If $E$ is a vector bundle of rank $r \geq 1$ over $X_K$ for some field $K \supseteq k$, then
  \begin{equation*}
    \ed_{k( E)}( E) \leq r - 1.
  \end{equation*}
\end{corollary}
\begin{proof}
  In the above, we take for $\calG$ the residue gerbe of the point given by $E$.
  Then the chosen coherent sheaf $F$ over $X_l$ is a vector bundle of rank $r$.
  According to Wedderburn's Theorem, we have
  \begin{equation*}
    \End( \pi_* F)/\frakj( \pi_* F) \cong \prod_i A_i \quad\text{with}\quad A_i \cong \Mat_{n_i \times n_i}( D_i)
  \end{equation*}
  for some division algebras $D_i$ over $k( E)$. Using Lemma \ref{lem:decompose}, we obtain that
  \begin{equation*}
    \pi_* F \cong \bigoplus_i E_i^{\oplus n_i} \quad\text{with}\quad \End( E_i)/\frakj( E_i) \cong D_i
  \end{equation*}
  for some vector bundles $E_i$ over $X_{k( E)}$.
  Corollary \ref{cor:ed_bound} and Lemma \ref{lem:connected} imply that
  \begin{equation*}
    \ed_{k( E)}( \Mod_{A_i, 1/d}) < \frac{n_i}{d} \dim_{k( E)} D_i \leq \frac{n_i}{d} \rank( E_i).
  \end{equation*}
  Using Proposition \ref{prop:product} and Proposition \ref{prop:reduce}, we conclude that
  \begin{equation*}
    \ed_{k( E)}( \Mod_{\End( \pi_* F), 1/d}) < \frac{1}{d} \rank( \pi_* F) = \rank( F) = r.
  \end{equation*}
  According to Corollary \ref{cor:ed(G)}, this means $\ed_{k( E)}( E) < r$, as claimed.
\end{proof}
\begin{corollary} \label{cor:elliptic}
  If $X$ is an elliptic curve over $k$, and $E$ is a vector bundle over $X_K$ for some field $K \supseteq k$,
  then $E$ is defined over its field of moduli $k( E) \subseteq K$.
\end{corollary}
\begin{proof}
  Hilbert's Nullstellensatz implies that some pullback of $E$ is already defined over some extension field
  of finite degree over $k( E)$. Therefore, Corollary \ref{cor:ed(G)} allows us to assume without loss of
  generality that $K$ has finite degree over $k( E)$.

  Let $d$ denote the degree of $K$ over $k( E)$. Let $\pi: X_K \twoheadrightarrow X_{k( E)}$
  be the canonical projection. Lemma \ref{lem:decompose} and Proposition \ref{prop:elliptic} together imply that
  \begin{equation*}
    \End( \pi_* E)/\frakj( \pi_* E) \cong \prod_i \Mat_{n_i \times n_i}( K_i)
  \end{equation*}
  for some (commutative!) fields $K_i \supseteq k( E)$ and some integers $n_i \geq 1$.

  Now we use Theorem \ref{thm:equivalence}, Proposition \ref{prop:M_R}, and Corollary \ref{cor:lift}.
  They allow us to conclude that since $E$ is defined over $K$, each integer $n_i$ is divisible by $d$,
  and therefore $E$ is already defined over $k( E)$.
\end{proof}

\section{Moduli of Sheaves with Nilpotent Endomorphisms} \label{sec:moduli}

As before, $X$ is a projective scheme over a field $k$. Let $\Nil_X^n$ denote the moduli stack of
coherent sheaves $E$ over $X$ and morphisms $\varphi: E \to E$ with $\varphi^n = 0$.
This stack is given by the following groupoid $\Nil_X^n( S)$ for each $k$-scheme $S$:
\begin{itemize}
 \item An object in $\Nil_X^n( S)$ consists of a coherent sheaf $\E$ over $X_S$ and a morphism
  $\varphi: \E \to \E$ with $\varphi^n = 0$ such that $\E$ and
  all $\coker( \varphi^i)$ are flat over $S$.
 \item A morphism in $\Nil_X^n( S)$ from $(\E, \varphi)$ to $(\F, \psi)$ is an isomorphism
  of coherent sheaves $\alpha: \E \to \F$ with $\alpha \circ \varphi = \psi \circ \alpha$.
\end{itemize}
If $(\E, \varphi)$ is an object in $\Nil_X^n(S)$, then $\im( \varphi)/\im( \varphi^i)$ is flat over $S$ for
each $i$, because $\E/\im( \varphi)$ and $\E/\im( \varphi^i)$ are so by assumption. The forgetful $1$-morphism
\begin{equation*}
  \Nil_X^n \longrightarrow \Coh_X, \qquad (\E, \varphi) \longmapsto \E
\end{equation*}
is representable and is of finite type. Therefore $\Nil_X^n$ is an Artin stack, and it is locally
of finite type over $k$. We have $\Nil_X^0 = \Spec( k)$ and $\Nil_X^1 = \Coh_X$.
We will describe $\Nil_X^n$ for $n \geq 2$ using the following moduli stacks of extensions.
 
Let $\calM^{\map}_X$ be the moduli stack of morphisms $\varphi: E_1\to E_2$ of coherent sheaves over $X$.
It is given by the following groupoid $\calM^{\map}_X( S)$ for each $k$-scheme $S$:
\begin{itemize}
 \item An object in $\calM^{\map}_X( S)$ is a morphism $\varphi: \E_1\to \E_2$ of
  coherent sheaves over $X_S$ such that $\coker( \varphi)$, $\im( \varphi)$ and
  $\ker( \varphi)$ are all flat over $S$.
 \item A morphism in $\calM^{\map}_X( S)$ from $\varphi: \E_1\to \E_2$ to
  $\psi: \F_1\to \F_2$ consists of two
  isomorphisms $\alpha_i: \E_i\to \F_i$ such that $\alpha_2 \circ \varphi = \psi \circ \alpha_1$.
\end{itemize}
We can also view an object $\varphi: \E_1\to  \E_2$ in $\calM^{\map}_X( S)$ as a pair of extensions
\begin{equation*}
  0 \to \ker( \varphi) \to \E_1 \to \im( \varphi) \to 0, \qquad
  0 \to \im( \varphi) \to \E_2 \to \coker( \varphi) \to 0.
\end{equation*}
In particular, $\E_1$ and $\E_2$ are also flat over $S$. The forgetful $1$-morphism
\begin{equation*}
  \calM^{\map}_X \longrightarrow \Coh_X \times_k \Coh_X, \qquad (\varphi: \E_1 \to \E_2) \longmapsto (\E_1, \E_2)
\end{equation*}
is representable and is of finite type. Therefore $\calM^{\map}_X$ is an Artin stack, and it is locally
of finite type over $k$.

Let $\calM^{\sub}_X$ be the moduli stack of pairs $E_1 \subset E_2$ of coherent sheaves over $X$.
We can identify it with the open substack in $\calM^{\map}_X$ where $\varphi$ is injective.

Let $\calM^{\subs}_X$ be the moduli stack of triples $E_1 \subset E \supset E_2$ of coherent sheaves over $X$.
It is given by the following groupoid $\calM^{\subs}_X( S)$ for each $k$-scheme $S$:
\begin{itemize}
 \item An object in $\calM^{\subs}_X( S)$ is a triple $\E_1 \subset \E \supset \E_2$ of coherent
  sheaves over $X_S$ such that $\frac{\E}{\E_1 + \E_2}$,
  $\frac{\E_1 + \E_2}{\E_i} \cong \frac{\E_{3-i}}{\E_1 \cap \E_2}$ and
  $\E_1 \cap \E_2$ are all flat over $S$.
 \item A morphism in $\calM^{\subs}_X( S)$ from $\E_1 \subset \E\supset \E_2$ to $\F_1 \subset \F\supset \F_2$
  is an isomorphism of coherent sheaves $\alpha: \E\to \F$ with $\alpha( \E_i) = \F_i$ for both $i$.
\end{itemize}
The forgetful $1$-morphism
\begin{equation*}
  \calM^{\subs}_X \longrightarrow \Coh_X, \qquad (\E_1 \subset \E \supset \E_2) \longmapsto \E
\end{equation*}
is representable and locally of finite type. Therefore $\calM^{\subs}_X$ is an Artin stack, and it is locally
of finite type over $k$. We will use the natural $1$-morphisms
\begin{align*}
  \prcap: \calM^{\subs}_X & \longrightarrow \calM^{\sub}_X, & (\E_1 \subset \E \supset \E_2)
    & \longmapsto (\E_1 \cap \E_2 \subset \E_1) \qquad\text{and}\\
  \prto: \calM^{\subs}_X & \longrightarrow \calM^{\map}_X, & (\E_1 \subset \E \supset \E_2)
    & \longmapsto (\E_2 \hookrightarrow \E \twoheadrightarrow \E/\E_1).
\end{align*}

Let $\calM^{\filt}_X$ be the moduli stack of chains $E_1 \subset E_2 \subset E$ of coherent sheaves over $X$.
This is the open substack in $\calM^{\subs}_X$ defined by the condition $\E_1 \subset \E_2$.
\begin{theorem} \label{thm:Nil_X}
  The natural $1$-morphism
  \begin{equation*}
    \pr_n: \Nil_X^{n+1} \longrightarrow \Nil_X^n, \qquad
      (\E, \varphi) \longmapsto (\im \varphi, \varphi|_{\im \varphi})
  \end{equation*}
  is isomorphic to a composition of pullbacks of the natural $1$-morphisms
  \begin{align*}
    \prsub: \calM^{\sub}_X & \longrightarrow \Coh_X,
      & (\E_1 \subset \E_2) & \longmapsto \E_1 \qquad\text{and}\\
    \prext: \calM^{\filt}_X & \longrightarrow \calM^{\map}_X,
      & (\E_1 \subset \E_2 \subset \E) & \longmapsto (\E_2 \hookrightarrow \E \twoheadrightarrow \E/\E_1).
  \end{align*}
  More precisely, the commutative diagram of stacks
  \begin{equation} \label{eq:Nil_cartesian} \xymatrix{
    \Nil_X^{n+1} \ar[rrr]^-{(\E, \varphi) \mapsto (\ker \varphi \subset \E \supset \im \varphi)}
      \ar[d]_{\pr_n} &&& \calM^{\subs}_X \ar[d]^{\prto}\\
    \Nil_X^n \ar[rrr]^{(\F, \psi) \mapsto (\psi: \F \to \F)} &&& \calM^{\map}_X
  } \end{equation}
  is cartesian, and the fibered product of stacks
  \begin{equation*} \xymatrix{
    \calM^{\map}_X \times_{\Coh_X} \calM^{\sub}_X \ar[rr] \ar[d] && \calM^{\sub}_X \ar[d]^{\prsub}\\
    \calM^{\map}_X \ar[rr]^{(\psi: \F' \to \F) \mapsto \ker \psi} && \Coh_X
  } \end{equation*}
  makes the following commutative diagram of stacks cartesian as well:
  \begin{equation} \label{eq:M_cartesian} \xymatrix{
    \calM^{\subs}_X \ar[rrrrr]^{(\E_1 \subset \E \supset \E_2) \mapsto (\E_1 \subset \E_1 + \E_2 \subset \E)}
      \ar[d]_{\prto \times \prcap} &&&&& \calM^{\filt}_X \ar[d]^{\prext}\\
    \calM^{\map}_X \times_{\Coh_X} \calM^{\sub}_X \ar[rrrrr]^-{(\psi: \F' \to \F, \ker \psi \subset \F'')
      \mapsto (\psi + 0: \frac{\F' \oplus \F''}{\ker \psi} \to \F)} &&&&& \calM^{\map}_X.
  } \end{equation}
  Here $\tfrac{\F' \oplus \F''}{\ker \psi}$ is the pushout of $\F'$ and $\F''$
  along their common subsheaf $\ker \psi$.
\end{theorem}
\begin{proof}
  We start with diagram \eqref{eq:Nil_cartesian}. Let an object $( \F, \psi)$ in $\Nil_X^n( S)$ be given,
  together with an object $\E_1 \subset \E \supset \E_2$ in $\calM^{\subs}_X( S)$. Let
  \begin{equation*} \xymatrix{
    \F \ar[d]_{\alpha} \ar[rr]^{\psi} && \F \ar[d]^{\beta}\\
    \E_2 \ar@{^{(}->}[r] & \E \ar@{->>}[r] & \E/\E_1
  } \end{equation*}
  be an isomorphism between the two images in $\calM^{\map}_X( S)$. Then the composition
  \begin{equation*}
    \varphi: \E \twoheadrightarrow \E/\E_1 \xrightarrow{\beta^{-1}} \F
      \xrightarrow{\alpha} \E_2 \hookrightarrow \E
  \end{equation*}
  has image $\E_2$ and restriction $\varphi|_{\E_2} = \alpha \circ \psi \circ \alpha^{-1}$. Since $\E/\E_2$ and
  all $\coker( \psi^i)$ are flat over $S$, we conclude that all $\coker( \varphi^i)$ are flat over $S$ as well.

  Therefore $( \E, \varphi)$ is an object in $\Nil_X^{n+1}( S)$ which gives back the given objects.
  This construction is functorial and shows that the diagram \eqref{eq:Nil_cartesian} is cartesian.

  It remains to consider diagram \eqref{eq:M_cartesian}. Let an object $\E_1 \subset \E_3 \subset \E$ in
  $\calM^{\filt}_X( S)$ be given, together with an object $\psi: \F'\to \F$ in $\calM^{\map}_X( S)$ and an object
  $\ker \psi \subset \F''$ in $\calM^{\sub}_X( S)$ that have the same image $\ker \psi$ in $\Coh_X( S)$. Let
  \begin{equation*} \xymatrix{
    \frac{\F' \oplus \F''}{\ker \psi} \ar[d]_{\alpha' + \alpha''} \ar[rr]^{\psi + 0} && \F \ar[d]^{\beta}\\
    \E_3 \ar@{^{(}->}[r] & \E \ar@{->>}[r] & \E/\E_1
  } \end{equation*}
  be an isomorphism between the two images in $\calM^{\map}_X( S)$.

  Comparing the cokernels and the kernels of the two horizontal maps, we see that $\E/\E_3 \cong \coker \psi$
  and $\E_1 = \alpha''( \F'')$ in $\E_3$. We put $\E_2 := \alpha'( \F')$ in $\E_3$.

  As $\alpha' + \alpha''$ is an isomorphism from the pushout of $\F'$ and $\F''$ to $\E_3$, the diagram
  \begin{equation*} \xymatrix{
    \ker \psi \ar@{^{(}->}[r] \ar@{^{(}->}[d] & \F'' \ar[d]^{\alpha''}\\
    \F' \ar[r]^{\alpha'} & \E_3
  } \end{equation*}
  is cocartesian.
  This implies that $\E_3 = \E_1 + \E_2$, and $\E_1 \cap \E_2 = \alpha'( \ker \psi) = \alpha''( \ker \psi)$.

  Therefore $( \E_1 \subset \E \supset \E_2)$ is an object in $\calM^{\subs}_X( S)$ which gives back the given
  objects. This shows that the diagram \eqref{eq:M_cartesian} is cartesian as well.
\end{proof}

Now let $C$ be a smooth projective curve of genus $g$ over the base field $k$. Then the stack $\Nil_{C, n}$
will turn out to be smooth of the expected dimension. In the very similar case of Higgs fields instead of
endomorphisms, Laumon has already proved this in \cite[Corollaire 2.10]{laumon:88}, using more local arguments.
\begin{corollary} \label{cor:moduli}
  The stack $\Nil_{C, n}$ is smooth over $k$. Its dimension at the $K$-valued point given by a coherent sheaf
  $E$ over $C_K$ and $\varphi \in \End( E)$ with $\varphi^n = 0$ is
  \begin{equation*}
    \dim_{(E, \varphi)} \Nil_{C, n} = (g-1) \sum_{i=1}^n r_i^2 ,
  \end{equation*}
  where $r_i$ denotes the rank of $\im( \varphi^{i-1})/\im( \varphi^i)$ over the generic point of $C_K$.
\end{corollary}
\begin{proof}
  We argue by induction over $n$, using the natural $1$-morphisms that appear in Theorem \ref{thm:Nil_X}.
  Given coherent sheaves $E_1$ and $E_2$ over $C_K$, we put
  \begin{equation*}
    \chi( E_2, E_1) := \dim_K \Hom( E_2, E_1) - \dim_K \Ext^1( E_2, E_1) \in \bbZ.
  \end{equation*}

  The natural $1$-morphism $\prsub$ is smooth of relative dimension
  \begin{equation*}
    - \chi( \tfrac{E_1}{E_0}, \tfrac{E_1}{E_0}) - \chi( \tfrac{E_1}{E_0}, E_0)
      = - \chi( \tfrac{E_1}{E_0}, E_1)
  \end{equation*}
  at any $K$-valued point $(E_0 \subset E_1)$, according to Proposition A.3 in \cite{hoffmann:10b}.

  The natural $1$-morphism $\prext$ is smooth of relative dimension
  \begin{equation*}
    - \chi( \tfrac{E}{E_1}, E_1) + \chi( \tfrac{E_3}{E_1}, E_1)
      = - \chi( \tfrac{E}{E_3}, E_1)
  \end{equation*}
  at any $K$-valued point $(E_1 \subset E_3 \subset E)$, according to Lemma 3.8 in \cite{bh:08}.

  Hence the natural $1$-morphism $\prto$ is smooth of relative dimension
  \begin{equation*}
    - \chi( \tfrac{E_1}{E_1 \cap E_2}, E_1) - \chi( \tfrac{E}{E_1 + E_2}, E_1)
      = -\chi( \tfrac{E}{E_2}, E_1)
  \end{equation*}
  at any $K$-valued point $( E_1 \subset E \supset E_2)$, due to the cartesian square \eqref{eq:M_cartesian}.

  Therefore the natural $1$-morphism $\pr_n$ is smooth of relative dimension
  \begin{equation*}
    - \chi( \tfrac{E}{\im \varphi}, \ker \varphi) = -\chi( \tfrac{E}{\im \varphi}, \tfrac{E}{\im \varphi})
      = ( g-1) r_1^2
  \end{equation*}
  at the $K$-valued point $( E, \varphi)$, due to the cartesian square \eqref{eq:Nil_cartesian}.
\end{proof}

\begin{corollary} \label{cor:r_i^2}
  Let $E$ be an indecomposable vector bundle over $C_K$ for an algebraically closed field $K \supseteq k$.
  If $r_i$ denotes the generic rank of $\im( \varphi^{i-1})/\im( \varphi^i)$ for a general element $\varphi$
  of the Jacobson radical $\frakj( E)$ in $\End( E)$, then
  \begin{equation*}
    \trdeg_k k( E) \leq 1 + (g-1) \sum_i r_i^2.
  \end{equation*}
\end{corollary}
\begin{proof}
  Since $E$ is indecomposable, Lemma 6 in \cite{atiyah:56} implies that $\End( E)/\frakj( E) \cong K$.

  Let $\calC \subseteq \Coh_X$ be the closure of the point given by $E$. It satisfies
  \begin{equation*}
    \dim_k \calC = \trdeg_k k( E) - \dim_K \End( E).
  \end{equation*}
  Choose $n \in \bbN$ with $\frakj( E)^n = 0$. Let $\calN \subseteq \Nil_{X, n}$ be
  the closure of all points $(E, \varphi)$ with $\varphi \in \frakj( E)$ such that
  each $\im( \varphi^{i-1})/\im( \varphi^i)$ has generic rank $r_i$. It satisfies
  \begin{equation*}
    \dim_k \calN \leq (g-1)( r_1^2 + \cdots + r_n^2)
  \end{equation*}
  due to Corollary \ref{cor:moduli}. The fiber of the forgetful $1$-morphism $\calN
\to\calC$ over the dense
  point $E: \Spec( K)\to \calC$ contains a dense open subscheme of $\frakj( E)$, so
  \begin{equation*}
    \dim_k \calN \geq \dim_k \calC + \dim_K \frakj( E) = \trdeg_k k( E) - 1. \qedhere
  \end{equation*}
\end{proof}
\begin{corollary} \label{cor:r^2-r}
  Let $E$ be a vector bundle of rank $r$ over $C_K$ for a field $K \supseteq k$. 
  If $E$ is not simple, and the curve $C$ has genus $g \geq 2$, then
  \begin{equation*}
    \trdeg_k k( E) \leq (g-1)(r^2 - r) + 2.
  \end{equation*}
\end{corollary}
\begin{proof}
  Since $k(E) = k(E \otimes_K L)$ for any field $L \supseteq K$, we may assume that $K$ is algebraically closed.
  We can express $E$ as a direct sum of some indecomposable vector bundles $E_j$ of rank $r_j \geq 1$
  over $C_K$. Corollary \ref{cor:r_i^2} states
  \begin{equation*}
    \trdeg_k k( E_j) \leq 1 + (g-1) \sum_i r_{ij}^2
  \end{equation*}
  for some integers $r_{ij} \geq 1$ with $\sum_i r_{ij} = r_j$. From this we conclude that
  \begin{equation*}
    \trdeg_k k( E) \leq \sum_j \trdeg_k k( E_j) \leq \sum_{j} 1 + (g-1) \sum_{i, j} r_{ij}^2.      
  \end{equation*}
  Because $E$ is not simple, the sum $\sum_{i, j} r_{ij} = r$ has at least two summands. Hence
  \begin{equation*}
    \trdeg_k k( E) \leq r + (g-1)(r^2 - 2r + 2) = (g-1)(r^2 - r) + 2 - (g-2)(r-2)
  \end{equation*}
  due to Lemma \ref{lem:arithm} below. Since $g \geq 2$ and $r \geq 2$, the result follows.
\end{proof}

\begin{lemma} \label{lem:arithm}
If $r_1, \ldots, r_n$, $n \geq 2$, are positive integers with $r_1 + \cdots + r_n = r$, then
  \begin{equation*}
    r_1^2 + \cdots + r_n^2 \leq r^2 - 2r + 2.
  \end{equation*}
\end{lemma}

\begin{proof}
  Since $r_2^2 + \cdots + r_n^2 \leq (r_2 + \cdots + r_n)^2$, it suffices to treat the case $n = 2$.
  In this case, the claim follows from $r_1^2 + r_2^2 = r^2 - 2 r_1 r_2$ and $r_1 r_2 \geq r-1$.
\end{proof}

\section{Essential Dimension of Vector Bundles over a Curve} \label{sec:bundles}
Let $C$ be a smooth projective irreducible curve of genus $g$ over the field $k$.
Assume that $C$ has a $k$-rational point. We consider the irreducible open substack
\begin{equation*}
  \Bun_{C, r, d} \subseteq \Coh_C
\end{equation*}
that parameterizes vector bundles of rank $r \geq 1$ and degree $d \in \bbZ$ over $C$. Let
\begin{equation*}
  \calG_{C, r, d}
\end{equation*}
denote the residue gerbe of the generic point on $\Bun_{C, r, d}$.
\begin{proposition}
  If the curve $C$ has genus $g = 0$, then
  \begin{equation*}
    \ed_k( \calG_{C, r, d}) = \ed_k( \Bun_{C, r, d}) = 0.
  \end{equation*}
\end{proposition}
\begin{proof}
  The assumptions imply $C \cong \bbP^1$. Let $E$ be a vector bundle over $\bbP^1_K$
  for some field $K \supseteq k$. Grothendieck's splitting theorem states
  \begin{equation*}
    E \cong \calO_{\bbP^1_K}( n_1) \oplus \cdots \oplus \calO_{\bbP^1_K}( n_r)
  \end{equation*}
  for some $n_1, \ldots, n_r \in \bbZ$. Therefore $E$ is already defined over $k$, so $\ed_k(E) = 0$.
\end{proof}
\begin{proposition} \label{prop:g=1}
  If the curve $C$ has genus $g = 1$, then
  \begin{equation*}
    \ed_k( \calG_{C, r, d}) = \gcd( r, d) \qquad\text{and}\qquad \ed_k( \Bun_{C, r, d}) = r.
  \end{equation*}
\end{proposition}
\begin{proof}
  Let $E$ be a vector bundle of rank $r$ over $C_K$ for some field $K \supseteq k$. We have
  \begin{equation*}
    \ed_k( E) = \trdeg_k k( E)
  \end{equation*}
  due to Corollary \ref{cor:elliptic}. Since $k(E) = k(E \otimes_K L)$ for any field $L \supseteq K$,
  we may assume without loss of generality that $K$ is algebraically closed.

  Suppose that $E$ is generic, or in other words that $E$ is an object of $\calG_{C, r, d}( K)$. Then $E$
  is semistable. Taking $K$ as base field, the image $T( E)$ under the Fourier-Mukai equivalence $T$
  in Theorem \ref{thm:FM} is a generic coherent torsion sheaf of length $\gcd( r, d)$ over $C_K$.
  This shows $\dim_K \End( E) = \gcd( r, d)$, and hence
  \begin{equation*}
    \ed_k( E) = \trdeg_k k( E) = \dim_k( \Bun_{C, r, d}) + \dim_K \Aut( E) = \gcd( r, d),
  \end{equation*}
  in this case. Consequently, we have $\ed_k( \calG_{C, r, d}) = \gcd( r, d)$. 

  In general, we can express $E$ as a direct sum of some indecomposable vector bundles $E_i$ over $C_K$.
  Corollary \ref{cor:r_i^2} states
  \begin{equation*}
    \trdeg_k k( E_i) \leq 1
  \end{equation*}
  for all $i$. This implies that $\trdeg_k k( E) \leq r$, and hence $\ed_k( E) \leq r$.

  We construct a vector bundle $E$ of given rank $r$ and degree $d$ with $\ed_k( E) = r$ as follows.
  Choose integers $d_1 < \ldots < d_r$ such that $d_1 + \cdots + d_r = d$. Let
  \begin{equation*}
    K = k \big( \prod_i \Pic^{d_i}( C) \big)
  \end{equation*}
  denote the function field of the product of the Picard varieties $\Pic^{d_i}( C)$. We choose a Poincar\'{e}
  bundle over $C \times_k \Pic^{d_i}( C)$ and denote its pullback to $C_K$ by $L_i$. Then
  \begin{equation*}
    E := L_1 \oplus \cdots \oplus L_r
  \end{equation*}
  is a vector bundle of rank $r$ and degree $d$ over $C_K$. We claim that $k( E) = K$.

  Let $S$ be an affine scheme over $K \otimes_k K$. Let $p, q: S \to \Spec( K)$ denote the two projections.
  Let an isomorphism of vector bundles over $C_S$
  \begin{equation*}
    \varphi: p^* E \longrightarrow q^* E
  \end{equation*}
  be given. If $i < j$, then $\deg( L_i) = d_i < d_j = \deg( L_j)$, and therefore
  \begin{equation*}
    \Hom( p^* L_i, q^* L_j) = 0.
  \end{equation*}
  This implies $\varphi( p^* E_i) \subseteq q^* E_i$ for the subbundle
  \begin{equation*}
    E_i := L_1 \oplus \cdots \oplus L_i \subseteq E.
  \end{equation*}
  Consequently, $\varphi$ induces for each index $i$ a morphism
  \begin{equation*}
    \varphi_i: p^* E_i/p^* E_{i-1} \longrightarrow q^* E_i/q^* E_{i-1}.
  \end{equation*}
  As $\varphi$ is an isomorphism, each $\varphi_i$ is an isomorphism. Using $E_i/E_{i-1} \cong L_i$, we get
  \begin{equation*}
    p^* L_i \cong q^* L_i
  \end{equation*}
  for each $i$. By construction of $K$, this implies $p = q$.
  In other words, the projection $S \to \Spec( K \otimes_k K)$ factors through the diagonal embedding
  \begin{equation*}
    \Spec( K) \hookrightarrow \Spec( K \otimes_k K).
  \end{equation*}

  Now let $\calG$ be the residue gerbe of the point in $\Bun_{C, r, d}$ given by $E$.
  This gerbe is given by the groupoid $[ I \rightrightarrows \Spec( K)]$, where
  \begin{equation*}
    I := \Spec( K) \times_{\Bun_{C, r, d}} \Spec( K)
  \end{equation*}
  parameterizes isomorphisms between pullbacks of $E$. We have just seen that the projection
  $I \to \Spec( K \otimes_k K)$ factors through the diagonal. Therefore, the identity on $\Spec( K)$
  descends to a morphism $\calG \to \Spec( K)$. This shows
  \begin{equation*}
    k( E) = K,
  \end{equation*}
  as claimed. Thus we obtain
  \begin{equation*}
    \trdeg_k k( E) = \dim_k \big( \prod_{i=1}^r \Pic^{d_i}( C) \big) = r
  \end{equation*}
  and hence $\ed_k( E) = r$ in this case. This shows that $\ed_k( \Bun_{C, r, d}) = r$.
\end{proof}
\begin{theorem} \label{thm:ed_vb}
  If the curve $C$ has genus $g \geq 2$, then
  \begin{equation*}
    \ed_k( \calG_{C, r, d}) = \ed_k( \Bun_{C, r, d}) \leq (g-1)r^2 + 1 + \sum_{p|h} ( p^{v_p( h)} - 1)
  \end{equation*}
  for $h := \gcd( r, d)$. One has equality here if Conjecture \ref{conj:CKM} holds.
\end{theorem}
\begin{proof}
  Let $E$ be a vector bundle of rank $r$ and degree $d$ over $C_K$ for some field $K \supseteq k$.
  If $E$ is not simple, then Corollary \ref{cor:r^2-r} and Corollary \ref{cor:vb_bound} imply
  \begin{align*}
    \ed_k( E) & = \trdeg_k k( E) + \ed_{k( E)}( E)\\ & \leq (g-1)(r^2 - r) + 2 + (r-1) \leq (g-1)r^2 + 1.
  \end{align*}

  Now suppose that $E$ is simple. Then Corollary \ref{cor:r^2-r} implies that
  \begin{equation} \label{eq:trdeg}
    \trdeg_k k( E) \leq (g-1)r^2 + 1.
  \end{equation}
  Let $\calG$ denote the residue gerbe of the point $E: \Spec( K)\to \Bun_{C, r, d}$. The residue field
  of this point is $k( E)$. Since $E$ is simple, Corollary \ref{cor:ed(G)} implies that
  \begin{equation*}
    \ed_{k( E)}( \calG) = \ed_{k( E)}( \Mod_{A, 1/\deg A})
  \end{equation*}
  for some central simple algebra $A$ over $k( E)$. The index of $A$ divides $h = \gcd( r, d)$, because its
  Brauer class coincides by Example \ref{ex:Gm-gerbe} with the Brauer class $\psi_{\calG}$ of the $\Gm$-gerbe
  $\calG$, and $\ind \psi_{\calG}$ divides $h$ for example by Corollary 3.6 in \cite{hoffmann:10b}. Hence
  \begin{equation} \label{eq:ed}
    \ed_{k( E)}( \calG) \leq \sum_{p|h} ( p^{v_p( h)} - 1)
  \end{equation}
  according to Proposition \ref{prop:morita} and Corollary \ref{cor:ed_D}. This proves the inequality.

  Suppose moreover that $E$ maps to the generic point of $\Bun_{C, r, d}$. Then we have equality in
  \eqref{eq:trdeg}. Assuming Conjecture \ref{conj:CKM}, we also have equality in \eqref{eq:ed}, because
  $\ind \psi_{\calG} = h$ in this situation. For a proof of the latter, see Proposition 5.1 in \cite{dn:89},
  or Corollary 6.6 in \cite{hoffmann:07}, or Theorem 1.8 in \cite{bbgn:07}.
\end{proof}

\end{document}